\documentclass{article}
\usepackage[utf8]{inputenc}
\usepackage{amsmath}
\usepackage{amsfonts}
\usepackage{amssymb}
\usepackage{amsthm} 
\usepackage{csquotes}
\usepackage[english]{babel}
\usepackage{lmodern}

\makeatletter
\newcommand{\myfootnote}[2]{\begingroup
  \def\@makefnmark{}%
  \addtocounter{footnote}{-1}%
  \footnote{{#1} #2}%
  \endgroup}
\makeatother

\title{On five dimensional Sasakian Lie algebras with trivial center}

\author{E. Loiudice and A. Lotta}

\date{}

\newtheorem{theorem}{Theorem}
\newtheorem{remark}{Remark}
\newtheorem{lemma}{Lemma}
\newtheorem{corollary}{Corollary}

\DeclareMathOperator{\spn}{span}

\DeclareMathOperator{\End}{End}
\newcommand{\di}{\mbox{d}}

\begin{document}

\maketitle
\myfootnote{{\em Mathematics Subject Classification (2000)}: Primary 53C25, 22E60; Secondary 53C30, 53C35.}{}
\myfootnote{{\em Keywords and Phrases}: Sasakian Lie algebra, Sasakian $\varphi$-symmetric space, Contact metric $(k,\mu)$--space.}{}

\begin{abstract}
We show that every five-dimensional Sasakian Lie algebra with trivial center is $\varphi$-symmetric. 
Moreover starting from a particular Sasakian structure on the Lie group $SL(2,\mathbb{R})\times\text{Aff}(\mathbb{R})$ we obtain a family of 
contact metric $(k,\mu)$ structures whose Boeckx invariants assume all values less than $-1$.
\end{abstract}

\section{Introduction}
In this paper we study some geometric features of Lie groups endowed with a left--invariant Sasakian structure. We focus on the five dimensional case,
our basic tool being the complete classification of five dimensional Sasakian Lie algebras carried out by Andrada, Fino and Vezzoni in \cite{AndradaFinoVezz}.
These Lie algebras fall into two classes: those whose center is one dimensional and generated by the Reeb vector field $\xi$,  and those
having trivial center. 

Our results concern the centerless case, and our purpose is twofold. 
First of all we prove that, in this case, the simply connected Lie group $G$ corresponding to a Sasakian Lie algebra 
 is {\em always} a Sasakian $\varphi$-symmetric space. We remark that, in their classification of five dimensional Sasakian $\varphi$-symmetric Lie algebras carried out
in \cite{FinoCalv}, Calvaruso and Fino excluded the centerless Lie algebras, so our result fills this gap. Indeed, we study the 
canonical fibration $G\to G/H$ of the Sasakian space $G$ over a homogeneous K\"ahler manifold, showing that $G/H$ is always either a product of two one dimensional complex space forms, namely $\mathbb{CH}^1(\lambda)\times\mathbb{CH}^1(\mu)$,  or $\mathbb{CP}^1(\lambda)\times\mathbb{CH}^1(\mu)$, or it is a two
dimensional complex space form $\mathbb{CH}^2(\alpha)$.

Our second objective is to show that, starting from a particular Sasakian metric $g$ on the Lie group
$SL(2,\mathbb{R})\times\text{Aff}(\mathbb{R})$, with contact form $\eta$, by a suitable deformation of  $g$,
one obtains a one parameter family of left--invariant metrics $g_a$, $a>1$, associated to the same 
$\eta$, for which the Reeb vector field $\xi$ 
satisfies the $(k,\mu)$ curvature condition
\begin{equation}
\label{kmu}
R(X,Y)\xi=k(\eta(Y)X-\eta(X)Y)+\mu(\eta(Y)hX-\eta(X)hY),\qquad k,\mu\in\mathbb{R},
\end{equation}
where $2h$ denotes the Lie derivative of the structural tensor $\varphi$ along $\xi$. 
For more information about this widely studied class of associated metrics,
we refer the reader to Blair's book \cite[Chapter~7]{Blair2010Book}. 
The Boeckx invariants $I_a$ of these metrics exhaust all values in $(-\infty,-1)$. 
We observe that Boeckx himself in \cite{Boeckx} provided Lie group models for contact metric $(k,\mu)$ spaces with $I<-1$; 
we feel that our construction provides an alternative way to obtain these models (at least in dimension five),  which seems to be more geometrically transparent.
Indeed, the definition of the metrics $g_a$ relies on a general principle explored in \cite{Carriazo}, which provides $(k,\mu)$ structures starting from Sasakian ones endowed with additional geometric objects, namely a pair of conjugate, integrable Legendre distributions, whose Pang invariants 
are suitable multiples of the original Sasakian metric.
However, no specific examples of application of these principle were given in \cite{Carriazo}, so we found interesting to carry out explicitly this construction on our
Sasakian Lie groups.

\section{Preliminaries}
For general notions concerning contact metric geometry and for the notation we refer the reader to Blair's book \cite{Blair2010Book}.
It is well known that a locally simmetric Sasakian manifold must be a space of constant curvature \cite{okumura}. This fact motivated Takahashi to introduce the notion of Sasakian $\varphi$-symmetric space, which is an analogous notion of Hermitian symmetric space, see \cite{takahashi}. We recall some basic
facts concerning this concept.
Let $(M,\varphi,\xi,\eta,g)$ be a Sasakian manifold. Consider an open neighborhood $U$ of $x\in M$ such that the induced Sasakian structure on $U$, denoted with the same symbols, is regular. Let $\pi:U \rightarrow U/\xi$ be the corresponding fibration and $(J,\bar{g})$ the induced K\"ahlerian structure on $U/\xi$ \cite{ogiue}. 
Then $M$ is called a \emph{Sasakian locally $\varphi$-symmetric space} if $U/ \xi$ is always a Hermitian locally symmetric space. 

This notion can be characterized also using the concept of $\varphi$-geodesic symmetry. A \emph{$\varphi$-geodesic} of a Sasakian manifold $(M, \varphi,\xi,\eta,g)$ is a geodesic $\gamma$ of $M$ such that $\eta(\gamma ')=0$; while a \emph{$\varphi$-geodesic symmetry} of $M$ at $m \in M$ is a local diffeomorphism $\sigma_m$ at $m$ such that for each $\varphi$-geodesic $\gamma=\gamma(s)$, with $\gamma(0)$ in the trajectory of the integral curve of $\xi$ passing through $m$, we have that 
$$
\sigma_m(\gamma(s))=\gamma(-s)
$$
for each $s$.

\begin{theorem}
 A Sasakian manifold $M$ is a Sasakian locally $\varphi$-symmetric space if it admits at every point a $\varphi$-geodesic symmetry which is also a local automorphism of the structure.
\end{theorem}

Let $M$ be a Sasakian locally  $\varphi$-symmetric space. If any $\varphi$-geodesic symmetry of $M$ is extendable to a global automorphism of $M$, and if the Killing vector field $\xi$ generates a one parameter group of global transformations, then $M$ is called a \emph{Sasakian $\varphi$-symmetric space}.
The following theorems are proved in \cite{takahashi}.
\begin{theorem}
 A Sasakian globally $\varphi$-symmetric space is a principal $G^1$-bundle over a Hermitian globally symmetric space.
\end{theorem}
Here $G^1$ is a $1$-dimensional Lie group isomorphic to the one parameter group of transformations generated by $\xi$.
\begin{theorem}
 A complete and simply connected Sasakian locally $\varphi$-symmetric space is a globally $\varphi$-symmetric space.
\end{theorem}

\bigskip
Now let $G$ be a simply connected Lie group endowed with a left-invariant Sasakian structure  $(\varphi,\xi,\eta,g)$. 
Let $\mathfrak{g}$ be the Lie algebra of $G$, so that $\xi\in\mathfrak{g}$.
Then, by homogeneity, $\xi$ is regular (cf. \cite[Theorem~4]{BoothWang}) and the fibration of $G$ is the canonical projection $p:G\to G/H$, 
where $H\subset G$ is the analytic subgroup with $Lie(H)=\mathbb{R} \xi$. In this case, the induced K\"ahler metric $\bar g$ on
$G/H$ is the unique one making $p$ a Riemannian submersion; clearly it is $G$-invariant. Hence $G$ is a Sasakian $\varphi$-symmetric
space if and only if $G/H$ is a Hermitian symmetric space.

\section{Five-dimensional Sasakian Lie algebras} \label{section on 5-dim lie algebras}

A \emph{Sasakian Lie algebra} is a Lie algebra $\mathfrak{g}$ endowed with a quadruple $(\varphi,\xi,\eta,g)$ where $\varphi \in \End (\mathfrak{g})$, $\xi \in \mathfrak{g}$, $\eta \in \mathfrak{g}^*$ and $g$ is an inner product such that
\begin{equation}
\begin{aligned}
 &\eta(\xi)=1, \quad \varphi^2 = -Id+\eta \otimes \xi, \quad g(\varphi X,\varphi Y)=g(X,Y)-\eta(X)\eta(Y),   \\
  &g(X,\varphi Y)=\di \eta(X,Y), \quad N_{\varphi}=-2 \di \eta \otimes \xi,   
\end{aligned}
\end{equation}
where 
$$N_{\varphi}(X,Y):=\varphi^2 [X,Y]+[\varphi X,\varphi Y]-\varphi [\varphi X,Y]- \varphi [X,\varphi Y].$$

Sasakian Lie algebras have been classified in \cite{AndradaFinoVezz} treating separately the two cases of Lie algebras with trivial or 
nontrivial center. We recall the classification result concerning Sasakian Lie algebras with trivial center \cite[Theorem~13]{AndradaFinoVezz}:
\begin{theorem}
 Let $(\mathfrak{g},\varphi, \xi,\eta,g)$ be a five dimensional Sasakian Lie algebra with trivial center. Then $\mathfrak{g}$ is isomorphic to one of the following Lie algebras: the direct products $\mathfrak{sl}(2,\mathbb{R})\times \mathfrak{aff(\mathbb{R})}$, $\mathfrak{su}(2)\times \mathfrak{aff(\mathbb{R})}$, or the non-unimodular solvable Lie algebra $\mathbb{R}^2 \ltimes \mathfrak{h}_3$.
\end{theorem}
Here $\mathfrak{aff}(\mathbb{R})$ and $\mathfrak{h}_3$ denote respectively the Lie algebra of the Lie group of affine motions of $\mathbb{R}$ and the real three-dimensional Heisenberg Lie algebra.

\medskip

Let $(\mathfrak{g},\varphi, \xi,\eta,g)$ be a five dimensional Sasakian Lie algebra with trivial center. According to \cite[Section~ 3.2]{AndradaFinoVezz},
after performing a $\mathcal{D}$-homothetic deformation of the structure, there exists an orthonormal basis $\{ e_1, \dots ,e_5 \}$ of $\mathfrak{g}$ satisfying:
$$
\ker \eta =\spn \{e_1, \dots ,e_4 \}, \, e_5=\xi, \, \varphi (e_1)=-e_2, \, \varphi (e_3)=-e_4,
$$
and such that one of the following eight sets of bracket relations holds:
\begin{equation}\label{A1} \tag{A1}
\begin{aligned}
 &[e_1,e_2]=\frac{2}{c} e_1-2e_5, &[e_1,e_3]&=ce_4, &[e_1,e_4]&=-ce_3, &[e_1,e_5]&=0,\\
 &[e_2,e_3]=-fe_4, &[e_2,e_4]&=fe_3,&[e_2,e_5]&=0, &[e_3,e_4]&=-2e_5,\\
  &[e_3,e_5]=-e_4, &[e_4,e_5]&=e_3, 
\end{aligned} 
\end{equation}

\begin{equation}\label{A2} \tag{A2}
\begin{aligned}
 &[e_1,e_2]=-a e_1-b e_2 -2e_5, &[e_1,e_3]&=ce_4, &[e_1,e_4]&=-ce_3, \\
 &[e_1,e_5]=0, &[e_2\,e_3]&=-\frac{2+ac}{b}e_4, &[e_2,e_4]&=\frac{2+ac}{b}e_3, \\
 &[e_2,e_5]=0, &[e_3,e_4]&=-2e_5, &[e_3,e_5]&=-e_4, \\
 &[e_4,e_5]=e_3, 
\end{aligned} 
\end{equation}

\begin{equation}\label{B1} \tag{B1}
\begin{aligned}
 &[e_1,e_2]=-a e_1 -2e_5, &[e_1,e_3]&=\frac{2}{a}e_4, &[e_1,e_4]&=-\frac{2}{a}e_3,&[e_1,e_5]&=0,  \\
 &[e_2,e_3]=-fe_4, &[e_2,e_4]&=fe_3, &[e_3,e_4]&=-2e_5, &[e_2,e_5]&=0, \\  
 &[e_3,e_5]=e_4, &[e_4,e_5]&=-e_3,
\end{aligned} 
\end{equation}

\begin{equation}\label{B2} \tag{B2}
\begin{aligned}
 &[e_1,e_2]=-a e_1-b e_2 -2e_5, &[e_1,e_3]&=ce_4, &[e_1,e_4]&=-ce_3, \\
 &[e_1,e_5]=0, &[e_2,e_3]&=-\frac{ac-2}{b}e_4, &[e_2,e_4]&=\frac{ac-2}{b}e_3, \\ 
 &[e_2,e_5]=0, &[e_3,e_4]&=-2e_5, &[e_3,e_5]&=e_4,  \\
 &[e_4,e_5]=-e_3, 
\end{aligned} 
\end{equation}

\begin{equation}\label{A3} \tag{A3}
\begin{aligned}
 &[e_1,e_2]=-a e_1-2e_5, &[e_1,e_3]&=-\frac{2}{a}e_4, &[e_1,e_4]&=\frac{2}{a}e_3, \\
 &[e_1,e_5]=0, &[e_2,e_3]&=\frac{a}{2}e_3-fe_4, &[e_2,e_4]&=fe_3+\frac{a}{2}e_4,\\
 &[e_2,e_5]=0, &[e_3,e_4]&=-ae_1-2e_5, &[e_3,e_5]&=-e_4, \\
 &[e_4,e_5]=e_3, 
\end{aligned} 
\end{equation}

\begin{equation}\label{A4} \tag{A4}
\begin{aligned}
 &[e_1,e_2]=-a e_1-b e_2 -2e_5, &[e_1,e_3]&=-\frac{1}{2}be_3+ce_4, \\
 &[e_1,e_4]=-ce_3-\frac{1}{2}be_4, &[e_1,e_5]&=0,  \\
  &[e_2,e_3]=\frac{1}{2}ae_3-\frac{2+ac}{b}e_4, &[e_2,e_4]&=\frac{2+ac}{b}e_3+\frac{1}{2}ae_4,\\
 &[e_2,e_5]=0, &[e_3,e_4]&=-a e_1-b e_2-2e_5,\\
 &[e_3,e_5]=-e_4, &[e_4,e_5]&=e_3,  
\end{aligned} 
\end{equation}

\begin{equation}\label{B3} \tag{B3}
\begin{aligned}
 &[e_1,e_2]=-a e_1-2e_5, &[e_1,e_3]&=\frac{2}{a}e_4, &[e_1,e_4]&=-\frac{2}{a}e_3,  \\
 &[e_1,e_5]=0, &[e_2,e_3]&=\frac{a}{2}e_3-fe_4, &[e_2,e_4]&=fe_3+\frac{a}{2}e_4,\\
 &[e_2,e_5]=0, &[e_3,e_4]&=-ae_1-2e_5, &[e_3,e_5]&=e_4, \\
  &[e_4,e_5]=-e_3, 
\end{aligned} 
\end{equation}

\begin{equation}\label{B4} \tag{B4}
\begin{aligned}
 &[e_1,e_2]=-a e_1-b e_2 -2e_5, &[e_1,e_3]&=-\frac{1}{2}be_3+ce_4, \\
 &[e_1,e_4]=-ce_3-\frac{1}{2}be_4, &[e_1,e_5]&=0,  \\
 &[e_2,e_3]=\frac{1}{2}ae_3-\frac{ac-2}{b}e_4, &[e_2,e_4]&=\frac{ac-2}{b}e_3+\frac{1}{2}ae_4, \\
 &[e_2,e_5]=0, &[e_3,e_4]&=-a e_1-b e_2-2e_5,  \\
  &[e_4,e_5]=-e_3, &[e_3,e_5]&=e_4.
\end{aligned} 
\end{equation}
\vspace{2mm}

In the first two cases we have that the Lie algebra $\mathfrak{g}$ is isomorphic to the direct product $\mathfrak{sl}(2,\mathbb{R})\times \mathfrak{aff(\mathbb{R})}$, where 
\begin{equation*}
\begin{cases}
 \mathfrak{aff}(\mathbb{R})\simeq \spn \{ fe_1+ce_2,e_1+ce_5 \}\\
 \mathfrak{sl}(2,\mathbb{R}) \simeq \spn \{e_3,e_4,e_5 \}
 \end{cases}
 \end{equation*}
 in case \eqref{A1}, and
 \begin{equation*}
\begin{cases}
  \mathfrak{aff}(\mathbb{R})\simeq \spn \{ ae_1+be_2+2e_5,e_1-ce_5 \}\\
  \mathfrak{sl}(2,\mathbb{R}) \simeq \spn \{e_3,e_4,e_5 \}
 \end{cases}
\end{equation*}
in case \eqref{A2}.

In cases \eqref{B1} and \eqref{B2} we have respectively that
\begin{equation*}
\begin{cases}
 \mathfrak{aff}(\mathbb{R})\simeq \spn \{ ae_1+e_5,fe_1+e_5 \}\\
 \mathfrak{su}(2) \simeq \spn \{e_3,e_4,e_5 \}
 \end{cases}
 \end{equation*}
 or
 \begin{equation*}
\begin{cases}
  \mathfrak{aff}(\mathbb{R})\simeq \spn \{ ae_1+be_2+2e_5,e_1+ce_5 \}\\
  \mathfrak{su}(2) \simeq \spn \{e_3,e_4,e_5 \}
 \end{cases}
\end{equation*}
and thus $\mathfrak{g}$ is isomorphic to the direct product $\mathfrak{su}(2)\times \mathfrak{aff(\mathbb{R})}$.
Lastly, in cases \eqref{A3}, \eqref{A4}, \eqref{B3}, \eqref{B4} one has $\mathfrak{g}\simeq \mathbb{R}^2 \ltimes \mathfrak{h}_3$.

\section{Sasakian $\varphi$-symmetric five-dimensional Lie groups}

In this section we prove our main result:

\begin{theorem}
Let $G$ be a five-dimensional simply connected Lie group endowed with a  left-invariant Sasakian structure. If the Lie algebra of $G$ has trivial center, then $G$ is a Sasakian $\varphi$-symmetric space.
\end{theorem}

\bigskip
We need the following:
 
\begin{lemma}\label{LemmaDistribution}
Let $(G/H,g)$ be a homogeneous Riemannian manifold, where $G$ is a connected Lie group and 
suppose we are given a reductive decomposition:
$$
\mathfrak{g}=\mathfrak{h}\oplus \mathfrak{m},
$$ 
of the Lie algebra $\mathfrak{g}$ of $G$. We suppose there exists an $Ad(H)$-invariant subspace $\mathfrak{n}$ of $\mathfrak{m}$ satisfying:
\begin{equation} 
\label{U}
[\mathfrak{g}, \mathfrak{h}\oplus \mathfrak{n}] \subset \mathfrak{h}\oplus \mathfrak{n}, \quad U(\mathfrak{m}, \mathfrak{n}) \subset \mathfrak{n},
\end{equation}
where $U:\mathfrak{m}\times \mathfrak{m}\rightarrow \mathfrak{m}$ is the bilinear operator defined by 
$$
2\langle U(X,Y),Z\rangle=\langle [Z,X]_{\mathfrak{m}}, Y\rangle+ \langle [Z,Y]_{\mathfrak{m}}, X\rangle. 
$$
Here $\langle\,,\,\rangle$ is the $Ad(H)$-invariant scalar product on $\mathfrak{m}$ corresponding to the Riemannian metric $g$ on $G/H$. 
Then the $G$-invariant distribution $\mathcal{D}$ determined by $\mathfrak{n}$ is parallel with respect to the Levi-Civita connection of $g$. 
\end{lemma}
\begin{proof}
Let $\pi:G \rightarrow G/H$ be the canonical projection, $e\in G$ the identity element of $G$ and $\underline{o}=\pi(e)$. For every $a\in G$ we denote by $\tau_a$ the isometry of $G/H$ given by $\tau_a(bH)=abH$. First of all we show  that the fundamental vector field $X^*$ determined by a vector $X\in \mathfrak{n}$ is a section of $\mathcal{D}$. Indeed, for every $a \in G$, according to \eqref{U} we have:
$$
(Ad(a^{-1})X) \in \mathfrak{h} \oplus \mathfrak{n},
$$ 
hence
$$
(\di \tau_{a^{-1}})_{aH}(X^*_{aH})=(Ad(a^{-1})X)^*_{\underline{o}}=(\di \pi)_e(Ad(a^{-1})X) \in \mathcal{D}_{\underline{o}},
$$
and this ensures that  $X^*_{aH}\in\mathcal{D}_{aH}$.
Now, let $\{ e_1, \dots ,e_k \}$ be a basis of $\mathfrak{n}$. Then 
$$
\{ (e_1^*)_{aH}, \dots ,(e_k^*)_{aH} \}
$$ 
is a basis of $\mathcal{D}_{aH}$, for every point $aH$ in a neighborhood $W$ of $\underline{o}$. Using the standard formula for the Levi-Civita
connection of a homogeneous Riemannian reductive space \cite{KN}, and the assumption \eqref{U}, we get:
$$
\nabla_X e^*_i=-\frac{1}{2}[X, e_i]_{\mathfrak{m}}+U(X,e_i)\in\mathfrak{n},
$$
for every $X\in T_{\underline{o}}(G/H)$. Here we are identifying $T_{\underline{o}}(G/H)$ with $\mathfrak{m}$ via  $(\di \pi)_e|_{\mathfrak{m}}:\mathfrak{m}\rightarrow T_{\underline{o}}(G/H)$. 
Then we see that for every section $Y$ of $\mathcal{D}$ and $X\in T_{\underline{o}}(G/H)$, $\nabla_XY$ belongs to $\mathcal{D}_{\underline{o}}$.
Hence, by the $G$-invariance, we can conclude that $\mathcal{D}$ is parallel.
\end{proof}

 \smallskip
\begin{proof}[Proof of the theorem]
Let $(\varphi, \xi, \eta,g)$ be a left-invariant Sasakian structure on $G$. We denote by $H$ the one parameter subgroup generated by $\xi$, and with $\bar{g}$ the induced K\"ahler metric on $G/H$ via the canonical submersion $\pi:G \rightarrow G/H$. We denote by $\mathfrak{h}=\mathbb{R}\xi$ the Lie algebra of $H$. As we recalled in Section~\ref{section on 5-dim lie algebras}, up to a $\mathcal{D}$-homothetic deformation, there exists an orthonormal basis $\{e_1, \dots ,e_5 \}$ of the Lie algebra $\mathfrak{g}$ of $G$ such that
$$
\ker \eta =\spn \{e_1, \dots ,e_4 \}, \, e_5=\xi, \, \varphi (e_1)=-e_2, \, \varphi (e_3)=-e_4,
$$
and the elements $e_1, \dots,e_5$ satisfy one of the bracket relations \eqref{A1}-\eqref{A4} or \eqref{B1}-\eqref{B4}. Moreover, $G/H$ is a homogeneous reductive space with a reductive decomposition given by 
\begin{equation*}
 \mathfrak{g}=\mathfrak{h}\oplus \mathfrak{m},
\end{equation*}
where $\mathfrak{m}:=\spn \{e_1,\dots ,e_4\}$.
The K\"ahler metric $\bar{g}$ on $G/H$ is the $G$-invariant metric induced by the scalar product $g_e$ on $\mathfrak{m}$. 

\medskip 

In cases \eqref{A1}, \eqref{A2}, \eqref{B1}, \eqref{B2} we consider the $G$-invariant distributions $\mathcal{D}_1$ and $\mathcal{D}_2$ on $G/H$ determined respectively by 
$$
\mathfrak{n}_1:= \spn \{e_1, e_2 \},\quad \mathfrak{n}_2:= \spn \{e_3 ,e_4 \}.
$$
Then $\mathcal{D}_1$ and $\mathcal{D}_2$ are orthogonal and complementary distributions on $G/H$ and one can easily check directly that $\mathfrak{n}_2$ satisfies the conditions \eqref{U} of Lemma~\ref{LemmaDistribution},
yielding that $\mathcal{D}_2$ is parallel. For example, in case \eqref{A1} we have:  
\begin{equation*}
\begin{aligned}
 &U(e_3,e_1)=\frac{1}{2}c e_4, &U(e_3, e_2)&=-\frac{1}{2}f e_4,  &U(e_3, e_3)&=0, &U(e_3,e_4)&=0, \\
 &U(e_4,e_1)=-\frac{1}{2}c e_3, &U(e_4, e_2)&=\frac{1}{2}f e_3,  &U(e_4,e_4)&=0.
\end{aligned} 
\end{equation*}
Then $G/H$ is isometric to the direct product $M_1 \times M_2$, where $M_1$ and $M_2$ are both $2$-dimensional homogeneous K\"ahler manifolds (cf.  \cite[Theorem~5.1, p. 211 and Theorem~8.1, p. 172]{KN}). 
Of course, each of these spaces has constant curvature, and this proves that $G/H$ is a Hermitian symmetric space. Actually in case \eqref{A1} and \eqref{A2} the sectional curvature of $M_1$  is $\lambda = -\frac{4}{c^2}$ or $\lambda =-(a^2+b^2)$ respectively, and the sectional curvature of $M_2$ is $\mu=-2$ in both cases; while $\lambda=-a^2$, $\mu =2$ in case \eqref{B1} and $\lambda=-(a^2+b^2)$, $\mu =2$ in case \eqref{B2}, as can be readily checked explicitly by using the corresponding bracket relations. 
Of course, this also yields that, as an Hermitian symmetric space, we have that 
$$
G/H\cong\mathbb{CH}^1(\lambda)\times\mathbb{CH}^1(\mu)
$$ 
in cases \eqref{A1} and \eqref{A2}, or
$$
G/H\cong \mathbb{CH}^1(\lambda) \times  \mathbb{CP}^1(\mu)
$$ 
in cases \eqref{B1} and \eqref{B2}. 

\medskip

Now consider the case \eqref{A3}. We set:  
$$
E_1=a e_1+2e_5, \, E_2=\frac{1}{a}e_2, \, E_3=e_3, \, E_4=e_4,
$$
and $\mathfrak{s}:=\spn \{E_1,\dots ,E_4 \}$. Then $\mathfrak{s}$ is an ideal 
of $\mathfrak{g}$ and 
$$\mathfrak{g}=\mathfrak{h}\oplus\mathfrak{s}.$$
Hence the analytic subgroup $S$ of $G$ corresponding to $\mathfrak{s}$ acts simply transitively on $G/H$ and
 the restriction $p:S\to G/H$ provides a diffeomorphism between $S$ and $G/H$. The $G$-invariant K\"ahler
 structure $(J,\bar g)$ of $G/H$ transfers via $p$ to a left-invariant K\"ahler structure on $S$, denoted by the same
 symbols. We claim  that $S$ is a complex hyperbolic space $\mathbb{C}H^2(\alpha)$.

Observe that $J$ is determined on $\mathfrak{s}$ by:
$$
J{E_1}=-a^2E_2, \quad {J}{E}_2=\frac{1}{a^2}{E}_1, \quad {J}{E}_3=-{E}_4, \quad {J}{E}_4={E}_3.
$$
Moreover, $\mathfrak{s}$ admits the following $\bar{g}$-orthogonal decomposition:
$$
\mathfrak{s}= \mathbb{R} A_0 \oplus\mathfrak{a}_1 \oplus \mathfrak{a}_2, 
$$
where 
$$
A_0=|a| E_2, \quad \mathfrak{a}_1=\spn \{E_3, E_4 \}, \quad \mathfrak{a}_2=\spn \{ JA_0 \}, \quad \lambda=\frac{|a|}{2}.
$$
The following relations hold:
\begin{equation*}
\begin{aligned}
 &[A_0, X]= \frac{|a|}{2} X+S_0(X),  &[A_0,JA_0]&=|a|JA_0\\
 &[X,Y]=|a| \bar{g}(JX,Y)JA_0, &[X, JA_0]&=0,
\end{aligned}
\end{equation*}
for every $X,Y \in \mathfrak{a}_1$, where $S_0:\mathfrak{a}_1+\mathfrak{a}_2\rightarrow \mathfrak{a}_1+\mathfrak{a}_2$ is the skew-symmetric derivation defined by 
$$S_0(E_3)=-\frac{af}{|a|}E_4, \quad S_0(E_4)=\frac{af}{|a|}E_3, \quad  S_0(JA_0)=0.$$ 
Now our claim follows from Heintze's description of the complex hyperbolic space as a solvable Lie group endowed with a left-invariant
K\"ahler metric of negative curvature (cf. \cite[Section 6]{Heintze}).

All the remaining cases can be treated similary. In case \eqref{A4} consider 
$$
E_1=a e_1+be_2+2e_5, \quad E_2=\frac{1}{b}e_1, \quad E_3=e_3, \quad E_4=e_4,
$$ 
and again the ideal $\mathfrak{s}=\spn \{E_1,\dots ,E_4 \}$. 
Now we have:
$$
JE_1=-\frac{a}{b}E_1+(a^2+b^2) E_2, \quad JE_2=-\frac{1}{b^2} E_1+\frac{a}{b}E_2, \quad JE_3=-E_4, \quad JE_4=E_3,
$$
and $\mathfrak{s}$ amdits the orthogonal decomposition:
$$
\mathfrak{s}= \mathbb{R} A_0 \oplus \mathfrak{a}_1 \oplus \mathfrak{a}_2, 
$$
with
\begin{equation}
 \begin{aligned}
  & A_0=\frac{1}{2 \lambda} \big(\frac{a}{b}E_1-(a^2+b^2)E_2 \big), \quad \lambda=\sqrt{\frac{a^2+b^2}{4}}\\
  &\mathfrak{a}_1=\spn \{E_3, E_4 \},   \quad \mathfrak{a}_2=\spn \{ JA_0 \}, 
 \end{aligned}
\end{equation}
Then the following relations hold 
\begin{equation*}
\begin{aligned}
 &[A_0, X]= \lambda X+S_0(X),  &[A_0,JA_0]&=2 \lambda JA_0\\
 &[X,Y]=2 \lambda \bar{g}(JX,Y)JA_0, &[X, JA_0]&=0,
\end{aligned}
\end{equation*}
for  $X,Y \in \mathfrak{a}_1$, 
where $S_0:\mathfrak{a}_1+\mathfrak{a}_2\rightarrow \mathfrak{a}_1+\mathfrak{a}_2$ is now given by 
$$S_0(E_3)=-2 \lambda \frac{c}{b}E_4, \quad S_0(E_4)=2 \lambda \frac{c}{b}E_3, \quad  S_0(JA_0)=0,$$
and thus again $S$ is a complex hyperbolic space.
Cases \eqref{B3} and \eqref{B4} are referable to cases \eqref{A3} and \eqref{A4} respectively. 
Indeed we can consider the ideals
$$
\mathfrak{s}=\spn \{ a e_1+2e_5, \frac{1}{a}e_2,  e_3, e_4 \}
$$
in case \eqref{B3} and  
$$
\mathfrak{s}=\spn \{a e_1+be_2+2e_5, \frac{1}{b}e_1, e_3, e_4\}
$$ 
in case \eqref{B4}, and the left-invariant K\"ahler structure on such Lie algebras obtained as before. 
One can check that in the first case this Lie algebra $\mathfrak{s}$ is isometric to the one considered in case \eqref{A3}, while in the second case $\mathfrak{s}$ is isometric to the Lie algebra considered in case \eqref{A4}, thus completing the proof.

\end{proof}

\section{$(k,\mu)$-contact metric spaces of dimension five}
In \cite{Boeckx} Boeckx classified all the non-Sasakian contact metric $(k,\mu)$-spaces, i.e., the
contact metric manifolds whose curvature satisfies condition \eqref{kmu},
up to $\mathcal{D}$-homothetic transformations. He showed that such a space is locally determined by its dimension and the number $I=\frac{1-\frac{\mu}{2}}{\sqrt{1-k}}$; moreover he gave explicit examples of $(k,\mu)$-spaces, consisting of some abstract Lie groups, for every possible dimension and and for every $I\leqslant -1$. 
We remark that in his exposition Boeckx does not give any explicit description of these Lie groups.
In this section we shall construct  $5$-dimensional non-Sasakian $(k,\mu)$-spaces with Boeckx invariant $I<-1$ in terms of classical Lie groups, using the classification of the $5$-dimensional Sasakian Lie algebras described in section \ref{section on 5-dim lie algebras}. Our tool will be the relation between Sasakian structures and non-Sasakian $(k,\mu)$-structures  explored by Cappelletti-Montano, Carriazo and Mart\'{\i}n-Molina in \cite{Carriazo}.  Their approach is based on the Pang invariant of a Legendre foliation 
(cf. \cite{Pang}). 
We recall that, if $\mathcal{D}$ is an integrable Legendre distribution on a contact manifold $(M,\eta)$,  
its Pang invariant is the symmetric bilinear form on the tangent bundle of $\mathcal{D}$, defined by
$$\Pi_{\mathcal{D}} (X, X') =-(\mathcal{L}_X\mathcal{L}_{X'} \eta)(\xi) = 2d\eta([\xi, X], X'),$$
where $\xi$ is the Reeb vector field.

\begin{theorem}
Let $\mathfrak{g}$ be the Sasakian Lie algebra of type \eqref{A2} with  
$$c=a=0,\quad b=-\sqrt{2}.$$ 
Let $G$ be a connected Lie group with Lie algebra $\mathfrak{g}$ and denote by $(\varphi, \xi, \eta, g)$ the associated left-invariant Sasakian structure on $G$.
Then $\mathfrak{g}$ admits an orthogonal decomposition
$$
\mathfrak{g}=\mathfrak{d}_+\oplus\mathfrak{d}_-\oplus \mathbb{R} \xi
$$
such that, for every real number $a>1$, the modified left-invariant metric $g_a$ defined by:
$$
(g_a)_e=
\begin{cases}
  \frac{1}{a} g_e,  &\text{on } \mathfrak{d}_+ \times \mathfrak{d}_+ \\
  a g_e ,   &\text{on } \mathfrak{d}_- \times \mathfrak{d}_-\\
   \eta_e \otimes \eta_e ,  &\text{otherwise}
 \end{cases},
$$
is also associated to the contact form $\eta$ and determines a contact metric $(k,\mu)$ structure on $G$, whose Boeckx invariant is
\begin{equation}
\label{Ia}
I_a=-\frac{a^2+1}{a^2-1}.
\end{equation}
\end{theorem}

\begin{proof}
Let $\{ e_1,\dots ,e_5 \}$ be the orthonormal basis of the Sasakian Lie algebra $(\mathfrak{g}, \varphi, \xi,\eta,g)$, with
\begin{equation}
\label{sasaki-A2}
\ker \eta =\spn \{e_1, \dots ,e_4 \}, \, e_5=\xi, \, \varphi (e_1)=-e_2, \, \varphi (e_3)=-e_4,
\end{equation}
and for which the nonzero bracket relations are:
\begin{equation*}
  \begin{aligned}
 &[e_1,e_2]=\sqrt{2} e_2 -2e_5, &[e_2\,e_3]&=\sqrt{2} e_4, &[e_2,e_4]&=-\sqrt{2} e_3, \\
 &[e_4,e_5]=e_3,  &[e_3,e_4]&=-2e_5, &[e_3,e_5]&=-e_4. 
\end{aligned} 
 \end{equation*}
Now, we observe that the following two Lie sub-algebras of $\mathfrak{g}$: 
$$
\mathfrak{d}_+:=\spn \{ e_1+e_4, e_2+e_3 \}, \quad \mathfrak{d}_-:=\spn \{ e_3-e_2, e_1-e_4 \}
$$
determine two totally geodesic, mutually orthogonal, Legendre left-invariant foliations $\mathcal{D_+}$ and $\mathcal{D_-}$ of $G$  with respect
to the Sasakian structure. Moreover, the Pang invariants $\Pi_{\mathcal{D_+}}$ and $\Pi_{\mathcal{D_-}}$ 
are determined at the neutral element of $G$ by:
\begin{equation*}
 \begin{aligned}
  & \Pi_{\mathcal{D_+}}(X,Y)=2 \di \eta ([\xi, X], Y)=-g(X,Y),\\
  & \Pi_{\mathcal{D_-}}(V,W)=2 \di \eta ([\xi, V], W)=-g(V,W),
 \end{aligned}
\end{equation*} 
where $X,Y\in\mathfrak{d}_+$ and $V,W\in\mathfrak{d}_-$. 
Hence, by the left-invariance of the Sasakian structure, Theorem~4.1 of \cite{Carriazo} is applicable, 
ensuring that there exists a family of contact metrics $(k_a,\mu_a)$-structure on $G$ compatible with the contact form $\eta$, with 
\begin{equation}
\label{tensors}
  g_a= 
 \begin{cases}
  -\frac{1}{a} g,  &\text{on } T \mathcal{D_+} \times T \mathcal{D_+} \\
  -a g ,   &\text{on } T \mathcal{D_-} \times T \mathcal{D_-}\\
   \eta \otimes \eta ,  &\text{otherwise}
 \end{cases},
 \quad \varphi _a= 
 \begin{cases}
  -\frac{1}{a} \varphi, &\text{on } T \mathcal{D_+} \\
  -a \varphi , &\text{on } T \mathcal{D_-}\\
  0,  &\text{on } \mathbb{R}\xi
 \end{cases},
\end{equation}
$$
k_a=1-\frac{(a^2-1)^2 }{16 a^2}, \quad  \mu_a=2-\frac{a^2+1}{2a},
$$
and $I_a=-\frac{a^2+1}{|a^2-1|}$, parametrized by $a< -1$. This yields the result.
\end{proof}

\bigskip
Since $I_a$ in \eqref{Ia} assumes all values in $(-\infty,-1)$, we can reformulate the above result in order to obtain the following local classification of
five dimensional $(k,\mu)$ contact metric spaces. Denote by $\{ E_1, \dots ,E_5 \}$ 
the standard basis of the Lie algebra $\mathfrak{sl}(2,\mathbb{R}) \times \mathfrak{aff(\mathbb{R})}$,
satisfying the following relations: 
$$[E_1,E_2]= E_2, \; [E_3,E_4]=-E_5, \; [E_3,E_5]=-2E_4, \; [E_4,E_5]=2E_3.$$

\begin{corollary}
Every five-dimensional contact metric $(k,\mu)$-space with Boeckx invariant $I<-1$ is locally equivalent,  up to a $\mathcal{D}$-homothetic deformation, to the direct product
$$SL(2,\mathbb{R}) \times Aff(\mathbb{R}),$$ 
endowed with the left invariant contact metric structure $(\varphi, \xi,\eta,g)$ determined with respect to the standard basis $\{ E_1, \dots ,E_5 \}$ of $\mathfrak{aff(\mathbb{R})} \times \mathfrak{sl}(2,\mathbb{R})$ by:
 \begin{equation*}
 \varphi \equiv 
 \begin{pmatrix}
       0        &   2s   &   t   &    0           &   0  \\
 -\frac{1}{2}s  &    0   &   0   & \frac{1}{2} t  &   0  \\
  \frac{1}{2}t  &    0   &   0   &    -s          &   0 \\
         0      &    t   &   s   &     0          &   0  \\
  \frac{1}{2}s  &    0   &   0   & -\frac{1}{2} t &   0     
 \end{pmatrix}, 
 \quad 
 g \equiv 
 \begin{pmatrix}
 -\frac{1}{2}s   &   0   &   0   &  \frac{1}{2} t   &   0  \\
  0   &   4-2s   &  -t   &   0   &   4  \\
  0   &   -t   &   -s    &    0   &   0 \\
 \frac{1}{2} t  &   0   &   0   &   -s    &   0  \\
  0   &   4   &   0   &   0   &   4     
 \end{pmatrix}, 
 \end{equation*}         
 $$\eta(E_1)=\eta(E_3)=\eta(E_4)=0, \quad \eta(E_2)=\eta(E_5)=2, \quad \xi=\frac{1}{2}E_5,$$
 where 
$$
s=-\frac{1}{2}(\sqrt{\frac{I-1}{I+1}}+\sqrt{\frac{I+1}{I-1}}), \quad 
t=-\frac{1}{\sqrt{2}}(\sqrt{\frac{I-1}{I+1}}-\sqrt{\frac{I+1}{I-1}}).$$

\end{corollary}

\bigskip

\begin{remark}\em
Concerning the already cited examples of $(k,\mu)$-spaces provided by Boeckx,  they consist in a $2$-parameter family of abstract Lie groups,
endowed with a left-invariant contact metric structure. 
Actually, one can check that, in the five dimensional case, the corresponing Lie algebras are all isomorphic to 
$\mathfrak{sl}(2,\mathbb{R}) \times \mathfrak{aff(\mathbb{R})}$. 
\end{remark}

\vspace{3 mm}
\noindent 
Eugenia Loiudice

\vspace{1 mm}
\noindent Dipartimento di Matematica, Università di Bari Aldo Moro, 

\noindent Via Orabona 4, 70125 Bari, Italy

\noindent \emph{e-mail}: eugenia.loiudice@uniba.it

\vspace{3 mm}
\noindent 
Antonio Lotta

\vspace{1 mm}
\noindent Dipartimento di Matematica, Università di Bari Aldo Moro, 

\noindent Via Orabona 4, 70125 Bari, Italy

\noindent \emph{e-mail}: antonio.lotta@uniba.it


\begin{thebibliography}{}

\bibitem{AndradaFinoVezz}
A.~Andrada, A.~Fino, and L.~Vezzoni:
\newblock {\em A class of Sasakian $5$-manifolds},
\newblock {Transform. Groups} {\bf 14} (2009), 493--512.


\bibitem{Blair2010Book}
D.~E. Blair:
\newblock {{Riemannian geometry of contact and symplectic manifolds}},
  volume 203 of {\em {Progress in Mathematics}}.
\newblock Birkh\"auser, Boston Inc., Boston, MA, second edition, 2010.


\bibitem{Boeckx}
E.~Boeckx:
\newblock {\em A full classification of contact metric $(k,\mu)$-spaces},
\newblock {Illinois J. Math.} {\bf 44} (2000), 212--219.

\bibitem{BoothWang}
W.~M. Boothby and H.~C. Wang:
\newblock {\em On contact manifolds},
\newblock { Ann. of Math. (2)} {\bf 68}  (1958), 721--734.


\bibitem{FinoCalv}
G.~Calvaruso and A.~Fino:
\newblock {\em Five-dimensional $K$-contact Lie algebras},
\newblock {Monatsh. Math.} {\bf 167} (2012), 35--59.


\bibitem{Carriazo}
B.~Cappelletti-Montano, A.~Carriazo, and V.~Martín-Molina:
\newblock {\em Sasaki-Einstein and paraSasaki-Einstein metrics from $(k,
  \mu)$-structures},
\newblock { J. Geom. Phys.} {\bf 73} (2013), 20--36.

\bibitem{Heintze}
E.~Heintze:
\newblock {\em On homogeneous manifolds of negative curvature},
\newblock {Math. Ann.} {\bf 211}  (1974), 23--34.

\bibitem{KN}
S.~Kobayashi and K~Nomizu:
\newblock {{Foundations of differential geometry. Vol. II}},
\newblock Interscience Publishers, 1969.

\bibitem{ogiue}
K.~Ogiue:
\newblock {\em On fiberings of almost contact manifolds},
\newblock {K\={o}dai Math. Sem. Rep.} {\bf 17} (1965), 53--62.

\bibitem{okumura}
M.~Okumura:
\newblock {\em Some remarks on space with a certain contact structure},
\newblock {T\"ohoku Math. J. (2)} {\bf 14} (1962), 135--145.

\bibitem{Pang}
M.~Y. Pang:
\newblock {\em The structure of Legendre foliations},
\newblock {Trans. Amer. Math. Soc.} {\bf 320}  (1990), 417--455.

\bibitem{takahashi}
T.~Takahashi:
\newblock {\em Sasakian $\phi$-symmetric spaces},
\newblock {T\"ohoku Math. J. (2)} {\bf 29} (1977), 91--113.

\end{thebibliography}
\end{document}